\documentclass[oneside,reqno,10pt]{amsart}

\usepackage[a4paper, 
            left=1.05in,
            right=1.05in,
            top=1.05in,
            bottom=1.05in,
            footskip=.25in]{geometry}

\usepackage[utf8]{inputenc}
\usepackage{microtype}
\usepackage{graphicx}
\usepackage{amssymb} 
\usepackage{hyperref}

\usepackage{thmtools} 
\usepackage{thm-restate}

\usepackage{microtype}

\theoremstyle{plain}
\newtheorem{thm}{}[section]
\newtheorem{lemma}[thm]{Lemma}
 
\newtheorem{observation}[thm]{Observation}%%%added by moritz

\newtheorem{theorem}[thm]{Theorem}

\theoremstyle{remark}

 \newtheorem{remark}[thm]{Remark}

\usepackage{centernot}  
\usepackage{xspace}

\usepackage{hyperref}

\usepackage[scr=boondoxo]{mathalpha}  
\newcommand{\coloneqq}{\mathrel{\mathop:}\mathrel{\mkern-1.2mu}=}

\newcommand{\struct}[1]{\mathfrak{#1}}

\newcommand{\age}{\ensuremath{\mathrm{age}}\xspace}

\newcommand{\Aut}
{\ensuremath{\mathrm{Aut}}\xspace}

\newcommand{\alg}[1]{\mathfrak{#1}}
\newcommand{\ind}[2]{\mathbf{1}_{#1}^{#2}}
\newcommand{\class}[1]{\mathcal{#1}}

\usepackage{enumitem} %to personalise enumerate and itemize
\usepackage{tikz} %for graphs
\usetikzlibrary{cd} %for commutative diagrams

\usepackage{mathtools} %for vcentcolon
\usepackage{svrsymbols} %for \atom
\newcommand{\blueone}{{\color{blue}1}}
\newcommand{\bluecdots}{{\color{blue}\cdots}}
\newcommand{\blueddots}{{\color{blue}\ddots}}
\newcommand{\N}{\mathbb{N}}
\newcommand{\Npos}{{\mathbb{N}^+}}
\newcommand{\proofpg}[1]{\textbf{#1}}

 \usepackage{pdfcomment}
  \usepackage{mathtools}

\begin{document}

\title[Answering a question of Barto\v{s}ov\'{a} and Scow]{The random ordered graph is a semi-retract of the canonically ordered atomless Boolean algebra}

\author{Michael Pinsker}
\author{Jakub Rydval}
\author{Moritz Sch\"{o}bi}
\author{Christoph Spiess}

\address{Institut f\"{u}r Diskrete Mathematik und Geometrie, FG Algebra, TU Wien, Austria}

\email{\{michael.pinsker,jakub.rydval,moritz.schoebi,christoph.spiess\}@tuwien.ac.at}

 \begin{abstract}   
 We prove that the random ordered graph is a semi-retract of the canonically ordered atomless Boolean algebra, hereby answering an open question of Barto\v{s}ov\'a and Scow.
\end{abstract}  

\thanks{\emph{Michael Pinsker, Jakub Rydval, and Moritz Sch\"{o}bi}: This research was funded in whole or in part by the Austrian Science Fund (FWF) [I 5948]. For the purpose of Open Access, the authors have applied a CC BY public copyright licence to any Author Accepted Manuscript (AAM) version arising from this submission. \emph{Michael Pinsker and Christoph Spiess}: This research is  funded by the European Union (ERC, POCOCOP, 101071674). Views and opinions expressed are however those of the author(s) only and do not necessarily reflect those of the European Union or the European Research Council Executive Agency. Neither the European Union nor the granting authority can be held responsible for them.}


\maketitle
  
\section{Introduction}

 \emph{Boolean algebras} are (universal) algebras with signature $\{\wedge,\vee,\neg,0,1\}$ satisfying axioms formalising basic properties of logical operators (conjunction, disjunction, and negation), see, e.g.,~\cite[Section~2.1.4]{bodirsky2021complexity}. 
Here, $\wedge$ and $\vee$ are function symbols of arity $2$, $\neg$ is of arity $1$ and $0,1$ are constant symbols. 
The class of all finite Boolean algebras fulfills the requirements of Fra\"iss\'e's theorem, and hence there exists a Fra\"iss\'e limit $\alg{B}$ of this class, commonly called the \emph{atomless Boolean algebra}.

Every Boolean algebra $\alg{A}=(A;\wedge,\vee,\neg,0,1)$ induces a partial order $\preceq_{\alg{A}}$ on its domain given by $a_1 \preceq_{\alg{A}} a_2\vcentcolon\Leftrightarrow a_1\vee a_2=a_2$. 
The direct successors of $0$, i.e., the elements $a\in A$ such that $a\wedge b\in\{0,a\}$ for all $b\in A$, are called \emph{atoms}.
An \emph{atomic} Boolean algebra is one where every non-zero element lies above an atom; an \emph{atomless} Boolean algebra contains no atoms at all. 
It is not hard to see that every finite Boolean algebra is atomic.
Given such a finite Boolean algebra $\alg{A}$, specifying an arbitrary linear order on its atoms and extending it anti-lexicographically to the unique decompositions of all elements as joins of a particular subset of its atoms allows us to linearly order the entire algebra in a way that extends $\preceq_{\alg{A}}$~\cite{kechris2005fraisse}.
Finite Boolean algebras ordered in this way are called \emph{naturally ordered.}
Also the class of all finite naturally ordered Boolean algebras fulfills the requirements of Fra\"iss\'e's theorem;
its Fra\"iss\'e-limit $(\alg{B},<)$ is the \emph{canonically ordered atomless Boolean algebra}~\cite{kechris2005fraisse}.\footnote{In~\cite{BARTOŠOVÁ_SCOW_2024} this structure is instead called the \emph{(countable) atomless Boolean algebra with a generic normal order}.}
As the notation suggests, removing the order from $(\alg{B},<)$ gives us back the countable atomless Boolean algebra $\alg{B}$ (up to isomorphism).

The \emph{random ordered graph} $(\struct{G},<)$ is the Fra\"{i}ss\'e limit of the class of all finite simple graphs with an additional linear ordering of their vertices; its reduct $\struct{G}=(G;E)$ is the well-known random graph, the Fra\"{i}ss\'e limit of the class of all finite simple graphs.
It is a folklore fact that both the random ordered graph and the canonically ordered atomless Boolean algebra have the \emph{Ramsey property}~\cite{kechris2005fraisse}. 

\subsection{A question of Barto\v{s}ov\'{a} and Scow.}
A \emph{semi-retraction} between two structures $\struct{C}$ and $\struct{D}$ is a pair of functions $\phi\colon C \rightarrow D$ and $\psi\colon D \rightarrow C$ acting on the quantifier-free types of $\struct{C}$ and $\struct{D}$ whose composition $\psi \circ \phi$ acts as the identity on the quantifier-free types of $\struct{C}$.
Over locally finite ordered structures, semi-retractions have been shown to transfer the Ramsey property~\cite{scow2021ramsey,BARTOŠOVÁ_SCOW_2024}.
They are closely related to the more general category-theoretical notion of \emph{pre-adjunctions}, which are in general known to transfer the Ramsey property~\cite{masulovic_pre-adjunctions_2018}.
Any semi-retraction gives rise to a pre-adjunction in a certain precise sense; in some cases, also the converse is true~\cite[Sections 6.2, 6.3]{BARTOŠOVÁ_SCOW_2024}. 
Regarding this converse, however, some important questions remain open. 
While any semi-retraction between two locally finite structures is known to induce a pre-adjunction between their ages (with embeddings as morphisms), the inverse statement is not \cite[Theorem 6.4, Question 6.6]{BARTOŠOVÁ_SCOW_2024}.

Barto\v{s}ov\'a and Scow~\cite[Theorem~4.1]{BARTOŠOVÁ_SCOW_2024} showed that the random graph is a semi-retract of the atomless Boolean algebra.
They then posed the question whether the same statement holds in the ordered setting which is relevant for the transfer of the Ramsey property~\cite[Question~4.3]{BARTOŠOVÁ_SCOW_2024}.
We answer this question in the affirmative, extending the proof of the unordered version of the statement~\cite[Theorem~4.1]{BARTOŠOVÁ_SCOW_2024}.

\begin{restatable}{theorem}{semiretraction}\label{thm:rg_semiret_aba}
  The random ordered graph is a semi-retract of the canonically ordered atomless Boolean algebra. 
  The functions witnessing this fact can be chosen to be order-preserving.
\end{restatable}

\section{Preliminaries}
\subsection{Structures.}
A \emph{signature} is a set $\tau$ containing function and relation symbols, each with an associated natural number called \emph{arity}.
We denote the arity of a function symbol $f$ by $ar(f)$, and do the same for relation symbols.
A $\tau$\emph{-structure} $\struct{C}$ is a set $C$, called \emph{domain}, together with \emph{operations} $f^\struct{C}\colon C^{ar(f)}\rightarrow C$ and \emph{relations} $R^\struct{C}\subseteq C^{ar(R)}$, corresponding to the symbols in the signature.
The signature $\tau$ is called \emph{relational} if it only contains relational symbols, and \emph{functional} if it purely consists of function symbols. 
In parallel, a $\tau$-structure is called a \emph{relational structure} or an \emph{algebra.}
A subset $D\subseteq C$ induces a \emph{substructure} of a $\tau$-structure $\struct{C}$ if it is closed under all functions $f^\struct{C}, f\in\tau$.
This substructure, which we denote by $\struct{D}$, then has functions $f^\struct{D}\coloneqq f^\struct{C}|_{D^{ar(f)}}$ and relations $R^\struct{D}\coloneqq R^\struct{C}\cap D^{ar(R)}$.
We write $\struct{D}\leq\struct{C}$ to indicate that $\struct{D}$ is a substructure of $\struct{C}$.
For any subset $D\subseteq C$, the substructure \emph{generated by} $D$ is the smallest substructure of $\struct{C}$ containing $D$; we denote it by $\langle D\rangle$.
If a substructure is generated by a finite set, we call it \emph{finitely generated}.
Write $\Npos\coloneqq \N\setminus\{0\}$. 
The \emph{union} $\bigcup_{n\in\Npos} \struct{C}_n$ of a family $(\struct{C}_n)_{n\in\Npos}$ of $\tau$-structures with the property that $\struct{C}_n$ is a substructure of $\struct{C}_{n+1}$ for every $n\in\Npos$ is the $\tau$-structure with domain $\bigcup_{n\in\Npos} C_n $, relations $R^{\cup_{n\in\Npos} \struct{C}_n}\coloneqq \bigcup_{n\in\Npos} R^{\struct{C}_n}$ and operations $f^{\cup_{n\in\Npos} \struct{C}_n}\coloneqq \bigcup_{n\in\Npos} f^{\struct{C}_n}$.
An \emph{expansion} of a $\tau$-structure $\struct{C}$ is a $\sigma$-structure $\struct{D}$ with the same domain $C$ such that $\tau\subseteq\sigma$ and $f^\struct{C}=f^\struct{D}$ as well as $R^\struct{C}=R^\struct{D}$ for all function symbols $f$ and relation symbols $R$ in $\tau$.
A \emph{restriction} of a $\tau$-structure $\struct{C}$ is a $\sigma$-structure $\struct{D}$ such that $\sigma\subseteq\tau$ and $\struct{C}$ is an expansion of $\struct{D}$.

Given two structures $\struct{C}$ and $\struct{D}$ of the same signature $\tau$, a \emph{homomorphism} is a function $g\colon C\rightarrow D$ preserving all operations and relations. 
That is, given a function symbol $f\in\tau$ and $c_1,\dots,c_{ar(f)}\in C$, we require that $g(f^\struct{C}(c_1,\dots,c_{ar(f)})=f^\struct{D}(g(c_1),\dots,g(c_{ar(f)}))$, and similarly, given a relation symbol $R\in\tau$ and $c_1,\dots,c_{ar(R)}\in C$, $(c_1,\dots,c_{ar(R)})\in R^\struct{C}\Rightarrow(g(c_1),\dots,g(c_{ar(f)}))\in R^\struct{D}$.
An \emph{embedding} is an injective homomorphism that also preserves the complements of relations, meaning that also $(c_1,\dots,c_{ar(R)})\in R^\struct{C}\Leftarrow(g(c_1),\dots,g(c_{ar(f)}))\in R^\struct{D}$.
If there is an embedding from $\struct{C}$ to $\struct{D}$, we say that $\struct{C}$ \emph{embeds} into $\struct{D}$, and write $\struct{C}\hookrightarrow\struct{D}$.
An \emph{isomorphism} is a surjective embedding, and an \emph{automorphism} is an isomorphism from a structure $\struct{C}$ to itself.

The automorphisms of a structure $\struct{C}$ are denoted by $\Aut(\struct{C})$.
For $k\in\Npos$ and $\bar{c},\bar{d}\in C^k$, we say that $\bar{d}$ is \emph{in the orbit of} $\bar{c}$ (under $\Aut(\struct{C))}$ if there is $\gamma\in\Aut(\struct{C})$ such that $\gamma(\bar{c})=\bar{d}$.
Clearly, this relation is symmetric, and even defines an equivalence relation on $C^k$. 
The equivalence classes of this relation are called the $\emph{orbits}$ (of $k$-tuples).
The structure $\struct{C}$ is called \emph{$\omega$-categorical} if for every $k\in\Npos$, there are only finitely many distinct orbits of $k$-tuples. 

\subsection{Logic.}\label{subs:logic}
We call a $\tau$-structure $\struct{C}$ \emph{homogeneous} if, for all finitely generated substructures $\struct{D}$ and $\struct{E}$ of $\struct{C}$ and every isomorphism $g\colon\struct{D}\rightarrow\struct{E}$, there is $\gamma\in\Aut(\struct{C})$ extending $g$.
It is not hard to see that if both $C$ and $\tau$ are of countable size, $\struct{C}$ being homogeneous implies that it is also $\omega$-categorical.

The \emph{type} of a $k$-tuple $\bar{c}$ over $C$ is the set of all $\tau$-formulas $\phi$ with free variables $x_1,\dots,x_k$ such that $\struct{C}\models\phi(\bar{c})$, and the \emph{quantifier-free type} of $\bar{c}$ is the set of all such formulas without quantifiers. 
If two tuples $(c_1,\dots,c_n)$ and $(d_1,\dots,d_n)$ are of the same quantifier-free type, the mapping $c_i\mapsto d_i$ extends to to an isomorphism $g$ of the substructures $\langle\{c_1,\dots,c_n\}\rangle$ and $\langle\{d_1,\dots,d_n\}\rangle$ (by exchanging every occurence of $c_i$ in any term by $d_i$). 
The extension is well-defined because if $s(c_1,\dots,c_n)=t(c_1,\dots,c_n)$ for any two $\tau$-terms $s$ and $t$, $s(x_1,\dots,x_n)=t(x_1,\dots,x_n)$ lies in the quantifier-free type of $(c_1,\dots,c_n)$, and therefore also that of $(d_1,\dots,d_n)$.
With a very similar reasoning, one can see that $g$ is bijective and both $g$ and its inverse $g^{-1}$ are homomorphisms, which implies that $g$ is an isomorphism.
The converse also holds: if there is an isomorphism $g\colon\langle\{c_1,\dots,c_n\}\rangle\rightarrow\langle\{d_1,\dots,d_n\}\rangle$ extending $c_i\mapsto d_i$, the two tuples are of the same quantifier-free type.
If two tuples are in the same orbit, they are of the same type. 
As one can easily see from the observations regarding quantifier-free types and isomorphisms we just made, in homogeneous structures, also the converse is true.

 Homogeneous structures arise as limit objects of certain well-behaved classes of finitely generated structures.
A class $\class{C}$ of finitely generated structures in a common signature $\tau$ is said to have the:
\begin{itemize}
    \item \emph{hereditary property (HP)} if, for all $\struct{C}\in\class{C}$ and all $\tau$-structures $\struct{D}\hookrightarrow\struct{C}$, also $\struct{D}\in\class{C}$;
    \item \emph{joint embedding property (JEP)} if, for all $\struct{C,D}\in\class{C}$, there is $\struct{E}\in\class{C}$ with $\struct{C},\struct{D}\hookrightarrow\struct{E}$;
    \item \emph{amalgamation property (AP)} if, for all $\struct{C},\struct{D}_1,\struct{D}_2\in\class{C}$ for which there exist embeddings $f_1\colon \struct{C} \hookrightarrow\struct{D}_1$ and $f_2\colon \struct{C}\hookrightarrow \struct{D}_2$, there exists $\struct{E}\in\class{C}$ and embeddings $g_1\colon \struct{D}_1  \hookrightarrow \struct{E}$ and $g_2\colon \struct{D}_2  \hookrightarrow \struct{E}$ such that $g_1\circ f_1 = g_2\circ f_2$.
\end{itemize}
The \emph{age} of a structure $\struct{C}$ consists of all structures isomorphic to a finitely generated substructure of $\struct{C}$. Every class $\class{C}$ of the form $\age(\struct{C})$ for a countable homogeneous structure $\struct{C}$ over a countable signature is (i) closed under isomorphisms, (ii) contains only countably many structures up to isomorphism, and (iii) satisfies the HP, JEP and AP.
Fra\"iss\'e's theorem states the converse:

\begin{theorem}[Fra\"iss\'e's theorem, Theorem 6.1.2 in~\cite{hodges_book}]
    Let $\class{C}$ be a class of finitely generated structures over a countable signature $\tau$ that is  
    \begin{enumerate}[label=\textnormal{(\roman*)}]
        \item closed under isomorphisms,
        \item contains only countably many structures up to isomorphism, and 
        \item satisfies the HP, JEP and AP.  
    \end{enumerate}
    Then there is an (up to isomorphism) unique countable homogeneous $\tau$-structure $\struct{C}$ with $\age(\struct{C})=\class{C}$, called the \emph{Fra\"iss\'e-limit} of $\class{C}$. 
\end{theorem}
Let $\struct{C}$ and $\struct{D}$ be two structures and let $\phi \colon C\rightarrow D$ be arbitrary.
We say that $\phi$ \emph{acts on} (or \emph{respects}) the quantifier-free types of $\struct{C}$ and $\struct{D}$ if for each two tuples $\bar{c}_1,\bar{c}_2$ over $C$ with the same quantifier-free type, the tuples $f(\bar{c}_1),f(\bar{c}_2)$ obtained via the component-wise action of $f$ also have the same quantifier-free type.\footnote{In $\omega$-categorical homogeneous structures, the notion of a quantifier-free type-respecting function coincides with the notion of a \emph{canonical function}.}
Given two structures $\struct{C}$ and $\struct{D}$, we say that $\struct{C}$ is a \emph{semi-retract} of $\struct{D}$ if there are quantifier-free type-respecting functions $\phi\colon C\rightarrow D$ and $\psi\colon D\rightarrow C$ such that $\psi\circ\phi$ acts as the identity on the quantifier-free types of $\struct{C}$.\footnote{Barto\v{s}ov\'a and Scow \cite{BARTOŠOVÁ_SCOW_2024} instead say that $\psi\circ\phi$ is \emph{quantifier-free type-preserving}.}
The pair $(\phi,\psi)$ is then called a \emph{semi-retraction} between $\struct{C}$ and $\struct{D}$.

\section{Main result}\label{sec:answer}

In the present section, we prove our main result Theorem~\ref{thm:rg_semiret_aba}, restated below.
\semiretraction* 

The two structures in question are typically constructed using Fra\"iss\'e's theorem; certain details of this construction were given in~\cite{kechris2005fraisse}; they will be useful in the proof of Theorem~\ref{thm:rg_semiret_aba}.
We will consider various powers of $\{0,1\}$. By interpreting $\wedge$ and $\vee$ as the component-wise infimum and supremum, respectively, $\neg$ as the operation switching all zeroes and ones, and the constant symbols $0$ and $1$ as the corresponding constant tuples, all powers of $\{0,1\}$ can be endowed with the structure of Boolean algebras. 
Note that, for every $I$, the elements of $\{0,1\}^I$ can be viewed as an indicator function $\ind{J}{I}\colon I \rightarrow \{0,1\}$, mapping $i$ to $1$ if and only if $i\in J$.
If $J$ consists of a single element $j$, we simply write $\ind{j}{I}$.
The following is obvious from the definition of the natural order on finite Boolean algebras \cite{kechris2005fraisse}.
\begin{lemma}\label{lemma:order_restriction}
    For a finite naturally ordered Boolean algebra $(\alg{D}_1,<_1)$ and a subalgebra $\alg{D}_2\leq\alg{D}_1$, the natural order of $\alg{D}_2$ induced by the restriction of $<_1$ to the atoms of $\alg{D}_2$ is the same as the restriction of $<_1$ to $\alg{D}_2$.
\end{lemma}

    Our proof of Theorem~\ref{thm:rg_semiret_aba} follows the proof of the aforementioned unordered statement.
To accommodate the added linear order, however, we need to take a diversion via another countably large but atomic Boolean algebra, whose finitely generated subalgebras can be ordered according to our needs.

\begin{proof}[Proof of Theorem~\ref{thm:rg_semiret_aba}]
    We construct order-preserving and quantifier-free type-respecting functions  $\phi\colon  (\struct{G};<) \rightarrow (\alg{B};<)$ and $\psi\colon  (\alg{B};<) \rightarrow (\struct{G};<)$.
    For the latter, we can proceed analogously to the construction in the unordered case~\cite{BARTOŠOVÁ_SCOW_2024}.%, which is why we start with this.

    \proofpg{Mapping the Boolean algebra to the random graph.}
    Expand $(\alg{B},<)$ by the binary (edge) relation $E$ defined as follows:
    $a,b\in B$ share an edge if and only if they are distinct and $a \wedge b \neq 0$.
    Note that the $\{<,E\}$-reduct $(B;<,E)$ of $(\alg{B},<,E)$ is a linearly ordered countable simple graph.
    Since the random ordered graph embeds every linearly ordered countable simple graph, there exists an embedding $\psi\colon (B;<,E) \hookrightarrow (\struct{G},<)$.
    By definition, this function is order-preserving and quantifier-free type-respecting from $(B;E,<)$ to $(\struct{G};<)$. 
    As $E$ is first-order definable in $(\alg{B},<)$ and $(\alg{B},<)$ is homogeneous, also $\psi\colon (\alg{B},<)\rightarrow(\struct{G},<)$ is order-preserving and quantifier-free type-respecting.
    
    \proofpg{The Boolean algebra of upper block triangular matrices.}
    The construction of the second order-preserving function $\phi\colon (G;<) \rightarrow (B;<) $, which is quantifier-free type-respecting from $(\struct{G},<)$ to $(\alg{B},<)$, is significantly more involved. 
    The basic idea stems from~\cite{BARTOŠOVÁ_SCOW_2024}, but to incorporate the additional order, we need to make the construction more explicit.
    Pick an arbitrary enumeration $(v_n)_{n\in\Npos}$ of $G$, and define $I\coloneqq \{(n,m)\in\Npos\times\Npos\mid n\leq m\}$. 
    Consider the natural Boolean algebra $\alg{U}$ induced by $\{0,1\}^I$; it will be helpful to think of its elements as $\Npos\times \Npos$-upper block triangular matrices.
    We map $(\struct{G},<)$ to $\alg{U}$ via $\theta\colon v_n\mapsto u_n$, where \[ u_n\coloneqq \ind{\{(n,n+k) \,\mid \, k\in\N\}}{I}\vee\ind{\{(m,n)\, \mid \, m<n \text{ and } (v_m,v_n) \in E^{\struct{G}}\}}{I}.\]
    Denote the subalgebra of $\alg{U}$ generated by $\theta(G)$ by $\alg{A}$.
    An exemplary configuration of $\{v_1,\dots,v_6\}$ and their images under $\theta$ is given in Figures $\ref{fig:six_vertices}$, \ref{fig:generators_and_atoms_1} and \ref{fig:generators_and_atoms_2}.
    
    \proofpg{The atoms of finitely generated subalgebras.}
    In $\alg{U}$, the elements $b_n\coloneqq \ind{\{(n,n+k)\mid k\in\N\}}{I},n\in\Npos$, form an infinite antichain, i.e., the meet of any two distinct elements is $0$.
    Moreover, below each $b_n$, the elements $b_n^i\coloneqq\ind{\{(n,i)\}}{I},i>n$, form another antichain.
    Using this notation, we see that $\theta$ sends each $v_n$ to 
    \[
    u_n= b_n\vee\bigvee\nolimits_{i<n \text{ and } (v_i,v_n)\in E^{\struct{G}}}b_i^n.
    \]
    This shows that the properties (1)--(4) identified in the proof of~\cite[Theorem~4.1]{BARTOŠOVÁ_SCOW_2024} also apply to the elements $u_n$ of $\alg{A}\leq\alg{U}$.
    For the convenience of the reader, we list them here. 
    \begin{enumerate}%[label=\textnormal{(\roman*)}]
        \item\label{it:proof:property1} Whenever $i,j$ and $k$ are pairwise distinct, we have that $u_i\wedge u_j\wedge u_k=0$. 
        \item Whenever $i<j\in\Npos$, we have that $u_i\wedge u_j$ is equal to $b_i^j=\ind{\{(i,j)\}}{I}$ if $(v_i,v_j)\in E^\struct{G}$, and equal to $0$ otherwise.
        \item For distinct $i$ and $j$, we have that
        \begin{equation*}
        u_i\wedge\neg u_j=\left\{
            \begin{array}{ll}
             u_i,&\text{if } (u_i,u_j)\notin E^\struct{G}, \\
             u_i\wedge\neg b_i^j,& \text{if } (u_i,u_j)\in E^\struct{G}\text{ and } i<j,\\
             u_i\wedge\neg b_j^i,& \text{if } (u_i,u_j)\in E^\struct{G}\text{ and } i>j.
        \end{array}\right. 
        \end{equation*}
        \item\label{it:proof:property4} For all $i$ and $j$, we have that $\neg u_i\wedge\neg u_j=\neg(u_i\vee u_j)$.
    \end{enumerate}
    These properties allow us to list the atoms of any subalgebra of $\alg{U}$ generated by finitely many $u_i$.
    The atoms of any finitely generated Boolean algebra are exactly the complete meet-expressions based on its generators not equal to zero.
    Consider a subalgebra generated by pairwise distinct elements $u_{i_1},\dots,u_{i_k}$.
    Making essentially the same observations Barto\v{s}ov\'{a} and Scow do in their proof of ~\cite[Theorem~4.1]{BARTOŠOVÁ_SCOW_2024} based off of the Properties (\ref{it:proof:property1})-(\ref{it:proof:property4}), its atoms (seen as matrices) are
    \begin{enumerate}[label=\textnormal{(\roman*)}]
        \item\label{it:proof:atom1} $\neg(\bigvee_{1\leq \ell\leq k} u_{i_\ell})$, the matrix that is constant zero in the lines $i_1,\dots,i_k$ and all entries with indices $(n,i_\ell), 1\leq\ell\leq k, n<i_\ell$ such that $(v_n,v_{i_\ell})\in E^{\struct{G}}$, and constant one otherwise,
        \item\label{it:proof:atom2} $u_{i_n}\wedge\neg(\bigvee_{1\leq \ell\leq k, \ell\neq n} u_{i_\ell}\wedge u_{i_n})\eqqcolon\Tilde{u}_{i_n}$, the matrix $u_{i_n}$ but with all entries with indices $(i_\ell,i_n), i_\ell<i_n$ and $(i_n,i_\ell), i_n<i_\ell$ such that $(v_{i_n},v_{i_\ell})\in E^\struct{G}$ set to zero, and
        \item\label{it:proof:atom3} $u_{i_n}\wedge u_{i_m}=b_{i_n}^{i_m}$, for $1\leq n,m\leq k$ such that $i_n<i_m$ and $(v_{i_n},v_{i_m})\in E^{\struct{G}}$.
    \end{enumerate}
    Exemplary atoms of the subalgebra generated by the images of the vertices $u_1,u_2$ and $u_4$ in Figure~\ref{fig:six_vertices} are shown in Figure~\ref{fig:generators_and_atoms_3}.
    In particular, this list tells us that given two tuples $\bar{v}=(v_{i_1},\dots,v_{i_n})$ and $\bar{v}'=(v_{j_1},\dots,v_{j_n)}$ of the same quantifier-free type in $\struct{G}$ (in particular, the same edges), the mapping $u_{i_\ell}\mapsto u_{j_\ell}$ canonically extends to a bijection between the atoms of the subalgebras $\langle\{u_{i_1},\dots,u_{i_n}\}\rangle$ and $\langle\{u_{j_1},\dots,u_{j_n}\}\rangle$, and therefore an isomorphism between the subalgebras. 
    By our observations on isomorphisms and quantifier-free types in Section~\ref{subs:logic}, we have: 
    \begin{observation}\label{observation:qftp-respecting}
        The mapping $\theta\colon (\struct{G},<)\rightarrow \alg{U}$ is quantifier-free type-respecting.
    \end{observation}
    Note that the addition of the order in the domain of $\theta$ further partitions the quantifier-free types, and thereby facilitates respecting them.
    
    \proofpg{Introducing a natural order.}
    The representation of the atoms of subalgebras generated by finitely many $u_i$ allows us to naturally order such algebras. 
    Consider pairwise distinct vertices $v_{i_1},\dots,v_{i_k}$.
    Without loss of generality, we have $v_{i_1}<v_{i_2}<\dots<v_{i_k}$.  
    Set $\neg(\bigvee_{1\leq \ell\leq k} u_{i_\ell})$ to be the biggest atom of $\langle\{u_{i_1},\dots,u_{i_k}\}\rangle$. 
    Directly below it, we order $\Tilde{u}_{i_1}<\dots<\Tilde{u}_{i_k}$, which are defined in \ref{it:proof:atom2}.
    The smallest atoms are those of the form $b_{i_n}^{i_m}$ as defined in \ref{it:proof:atom3}. 
    They are ordered according to 
    \begin{equation*}
    b_{i_n}^{i_m}<b_{i_{n'}}^{i_{m'}}\Leftrightarrow (\min(n,m),\max(n,m))<_{lex}(\min(n',m'),\max(n',m')),     
    \end{equation*}
    where $<_{lex}$ denotes is the lexicographic ordering.
    The complete ordered list of atoms of $\langle\{u_{i_1},\dots,u_{i_k}\}\rangle$ thus is
    \begin{equation}\label{eq:order_atoms}\tag{$\atom$}
        b_{[i_1,i_2]}<b_{[i_1,i_3]}<b_{[i_1,i_4]}<\dots<b_{[i_{k-1},i_k]}<\Tilde{u}_{i_1}<\dots<\Tilde{u}_{i_k}<\neg(\bigvee_{1\leq \ell\leq k} u_{i_\ell}),
    \end{equation}
    with $b_{[i,j]}$ standing for $b_{\min(i,j)}^{\max(i,j)}$, and only those $b_{[i,j]}$ corresponding to edges appearing in this list. 
    This ordering can be extended to a natural ordering of the entire subalgebra $\langle\{u_{i_1},\dots,u_{i_k}\}\rangle$.
    From now on, all finitely generated algebras we consider are generated by elements $u_i$.
    For every such algebra, $<$ denotes the natural ordering obtained in the way we just explored.

    Note that whenever we consider subalgebras generated by elements $u_{i_1},\dots,u_{i_k}$ and $u_{j_1},\dots,u_{j_\ell}$ such that $\{u_{j_1},\dots,u_{j_\ell}\}\subseteq\{u_{i_1},\dots,u_{i_k}\}$, the atoms of $\langle\{u_{j_1},\dots,u_{j_\ell}\}\rangle$ are ordered the same way in both $\langle\{u_{j_1},\dots,u_{j_\ell}\}\rangle$ and $\langle\{u_{i_1},\dots,u_{i_k}\}\rangle$.
    Thus, Lemma~\ref{lemma:order_restriction} tells us that 
    \[(\langle\{u_{j_1},\dots,u_{j_\ell}\}\rangle,<)\leq(\langle\{u_{i_1},\dots,u_{i_k}\}\rangle,<).\]

    Denote the subalgebra of $\alg{A}$ generated by $\{u_1,\dots,u_k\}$ by $\alg{A}_k$. 
    As we just saw, $(\alg{A}_k,<)$ is a substructure of $(\alg{A}_{k+j},<)$ for any $j\in\N$.
    This allows us to define the union $(\alg{A},<)\coloneqq\bigcup_{k\in\Npos}(\alg{A}_k,<)$, a countable, linearly ordered Boolean algebra.

    Given two tuples $\bar{v}=(v_{i_1},\dots,v_{i_k})$ and $\bar{v}'=(v_{j_1},\dots,v_{j_k})$ of the same quantifier-free type in $(\struct{G},<)$, Observation~\ref{observation:qftp-respecting} and the remarks we made right before it tell us that $\theta(\bar{v})$ and $\theta(\bar{v}')$ are of the same quantifier free type in $\alg{A}\leq\alg{U}$, and that  the mapping $\theta(v_{i_\ell})\mapsto \theta(v_{j_\ell})$ canonically extends to an isomorphism of the (unordered) subalgebras generated by $\theta(\bar{v})$ and $\theta(\bar{v}')$.
    In particular, both algebras have uniform (by exchanging $\theta(v_{i_\ell})$ for $\theta(v_{j_\ell})$ in the terms generating them; see \ref{it:proof:atom1}-\ref{it:proof:atom3}) sets of atoms.
    As the order imposed on said atoms is only dependent on that of their generators (in the way shown in (\ref{eq:order_atoms})), and $\bar{v}$ and $\bar{v}'$ are ordered the same way, so are the atoms of the subalgebras generated by their respective images. 
    Therefore, the canonical extension of $\theta(v_{i_\ell})\mapsto \theta(v_{j_\ell})$ is an isomorphism of the ordered subalgebras.
    In other words, $\theta'$, the corestriction of $\theta$ to $A$ with the added linear ordering, i.e. viewed as a mapping from $(\struct{G},<)$ to $(\alg{A},<)$, is quantifier-free type-respecting. 

    \proofpg{Putting everything together.}
    The canonically ordered atomless Boolean algebra $(\alg{B},<)$ embeds  $(\alg{A}_k,<)$ for every $k\in \Npos$ and is homogeneous.
    Consequently, it also contains an isomorphic copy of $(\alg{A},<)$.
    Fix an embedding $\iota$ from $(\alg{A},<)$ to $(\alg{B},<)$. 
    Now $\phi\coloneqq\iota\circ\theta'$ is quantifier-free type-respecting and order-preserving. 
    Finally, observe that $\phi$ maps distinct vertices to points with meet $0$ if and only if they do not share an edge. 
    Together with the definition of $\psi$ at the beginning of the proof and the fact that $\phi$ and $\psi$ both preserve the respective linear orders, this means that $\psi\circ\phi$ acts as the identity on the orbits of $(\struct{G},<)$.
\end{proof}

\begin{figure}[ht]
    \centering
    \begin{tikzpicture}[main/.style = {draw, circle}] 
        \node[main] (1) {$v_1$};
        \node[main] (2) [below of=1] {$v_2$};
        \node[main] (3) [right of=1] {$v_3$};
        \node[main] (4) [right of=2] {$v_4$};
        \node[main] (5) [right of=3] {$v_5$};
        \node[main] (5) [right of=3] {$v_5$};
        \node[main] (6) [right of=4] {$v_6$};

        \draw[-] (1) -- (2);
        \draw[-] (1) to[out=-45,in=135] (6);  
        \draw[-] (2) -- (3);
        \draw[-] (2) -- (4);    
        \draw[-] (3) -- (4);
        \draw[-] (4) -- (6);  
    \end{tikzpicture}
    \caption{A possible configuration of $v_1$ through $v_6$.}
    \label{fig:six_vertices}
\[
    \begin{array}{cccccccc}%u1
    \blueone&\blueone&\blueone&\blueone&\blueone&\blueone&\blueone&\bluecdots\\
    &0&0&0&0&0&0&\cdots\\
    &&0&0&0&0&0&\cdots\\
    &&&0&0&0&0&\cdots\\
    &&&&0&0&0&\cdots\\
    &&&&&0&0&\cdots\\
    &&&&&&0&\cdots\\
    &&&&&&&\ddots
  \end{array}
  \;
    \begin{array}{cccccccc}%u2
    0&\blueone&0&0&0&0&0&\cdots\\
    &\blueone&\blueone&\blueone&\blueone&\blueone&\blueone&\bluecdots\\
    &&0&0&0&0&0&\cdots\\
    &&&0&0&0&0&\cdots\\
    &&&&0&0&0&\cdots\\
    &&&&&0&0&\cdots\\
    &&&&&&0&\cdots\\
    &&&&&&&\ddots
  \end{array}
  \;
  \begin{array}{cccccccc}%u3
    0&0&0&0&0&0&0&\cdots\\
    &0&\blueone&0&0&0&0&\cdots\\
    &&\blueone&\blueone&\blueone&\blueone&\blueone&\bluecdots\\
    &&&0&0&0&0&\cdots\\
    &&&&0&0&0&\cdots\\
    &&&&&0&0&\cdots\\
    &&&&&&0&\cdots\\
    &&&&&&&\ddots
  \end{array}
  \]
  \caption{The tuples $u_1$, $u_2$ and $u_3$ in correspondence to the vertices in Figure~\ref{fig:six_vertices}.}\label{fig:generators_and_atoms_1}

    \centering
\[
   \begin{array}{cccccccc}%u4
    0&0&0&0&0&0&0&\cdots\\
    &0&0&\blueone&0&0&0&\cdots\\
    &&0&\blueone&0&0&0&\cdots\\
    &&&\blueone&\blueone&\blueone&\blueone&\bluecdots\\
    &&&&0&0&0&\cdots\\
    &&&&&0&0&\cdots\\
    &&&&&&0&\cdots\\
    &&&&&&&\ddots
  \end{array}
  \;
  \begin{array}{cccccccc}%u5
    0&0&0&0&0&0&0&\cdots\\
    &0&0&0&0&0&0&\cdots\\
    &&0&0&0&0&0&\cdots\\
    &&&0&0&0&0&\cdots\\
    &&&&\blueone&\blueone&\blueone&\bluecdots\\
    &&&&&0&0&\cdots\\
    &&&&&&0&\cdots\\
    &&&&&&&\ddots
  \end{array}
  \;
  \begin{array}{cccccccc}%u6
    0&0&0&0&0&\blueone&0&\cdots\\
    &0&0&0&0&0&0&\cdots\\
    &&0&0&0&0&0&\cdots\\
    &&&0&0&\blueone&0&\cdots\\
    &&&&0&0&0&\cdots\\
    &&&&&\blueone&\blueone&\bluecdots\\
    &&&&&&0&\cdots\\
    &&&&&&&\ddots
  \end{array}
\]
\caption{The tuples $u_4$, $u_5$ and $u_6$ in correspondence to the vertices in Figure~\ref{fig:six_vertices}.}\label{fig:generators_and_atoms_2}
\[
    \begin{array}{cccccccc}
    0&0&0&0&0&0&0&\cdots\\
    &0&0&0&0&0&0&\cdots\\
    &&\blueone&0&\blueone&\blueone&\blueone&\bluecdots\\
    &&&0&0&0&0&\cdots\\
    &&&&\blueone&\blueone&\blueone&\bluecdots\\
    &&&&&\blueone&\blueone&\bluecdots\\
    &&&&&&\blueone&\bluecdots\\
    &&&&&&&\blueddots
   \end{array}
   \;
   \begin{array}{cccccccc}
    0&0&0&0&0&0&0&\cdots\\
    &\blueone&\blueone&0&\blueone&\blueone&\blueone&\bluecdots\\
    &&0&0&0&0&0&\cdots\\
    &&&0&0&0&0&\cdots\\
    &&&&0&0&0&\cdots\\
    &&&&&0&0&\cdots\\
    &&&&&&0&\cdots\\
    &&&&&&&\ddots
  \end{array}
\;
  \begin{array}{cccccccc}
    0&0&0&0&0&0&0&\cdots\\
    &0&0&\blueone&0&0&0&\cdots\\
    &&0&0&0&0&0&\cdots\\
    &&&0&0&0&0&\cdots\\
    &&&&0&0&0&\cdots\\
    &&&&&0&0&\cdots\\
    &&&&&&0&\cdots\\
    &&&&&&&\ddots
  \end{array}
\]
\caption{Exemplary atoms of $\langle\{u_1,u_2,u_4\}\rangle$ in correspondence to the vertices in Figure~\ref{fig:six_vertices}: $\neg(\bigvee_{i=1,2,4} u_{i})$, $u_{2}\wedge\neg(\bigvee_{i=1,4} u_{i}\wedge u_{2})$ and $u_{2}\wedge u_{4}$.}\label{fig:generators_and_atoms_3}
\end{figure}

\begin{remark}
    The natural ordering of the finitely generated subalgebras $\alg{A}_k$ can also be achieved in a different manner. 
    As it offers some further insight, we sketch it here.

    Note that all entries of the elements of $A_k$ are fully determined by their first $k+1$ columns (considered as $\Npos\times \Npos$-upper block triangular matrices).
    Moreover, by cutting off all other entries, $\alg{A}_k$ naturally embeds into the Boolean algebra of all upper block triangular $(k+1)\times(k+1)$-matrices with entries $0$ and $1$, which we call $\alg{C}_{k+1}$. 
    
    Set $I_{k+1}\coloneqq\{(n,m)\in I\mid n,m\leq k+1\}$.
    The finite algebras $\alg{C}_{k+1}$ can be naturally ordered by ordering their atoms.
    Starting with the diagonal, we set $\ind{(k+1,k+1)}{I_{k+1}}$ to be the biggest atom of $\alg{C}_{k+1}$. Directly below it, we order $\left\{\ind{(1,1)}{I_{k+1}},\dots,\ind{(k,k)}{I_{k+1}}\right\}$ according to the ordering $<$ of $(\struct{G},<)$ restricted to $\{v_1,\dots, v_k\}$.
    Having ordered the diagonal, the remaining atoms $\ind{(i,j)}{I_{k+1}}, i\neq j$ are ordered as follows: 
     \begin{equation*}
        \ind{(i,j)}{I_{k+1}}<\ind{(k,l)}{I_{k+1}}\Leftrightarrow(\min\{v_{i,i},v_{j,j}\},\max\{v_{i,i},v_{j,j}\})<_{lex}(\min\{v_{k,k},v_{l,l}\},\max\{v_{k,k},v_{l,l}\}).
    \end{equation*}
    
    Pulling back the order from $(\alg{C}_{k+1},<)$ to $\alg{A}_k$ for all $k\in\Npos$, we obtain a natural ordering 
    $(\alg{A}_k,<)$ of the algebras $\alg{A}_k$. 
    In fact, this is the same ordering we imposed directly in the proof of Theorem~\ref{thm:rg_semiret_aba}.
\end{remark}

\section{Conclusion}

We answered an open question of Barto\v{s}ov\'{a} and Scow~\cite{BARTOŠOVÁ_SCOW_2024} regarding the existence of a semi-retraction between the random ordered graph and the canonically ordered atomless Boolean algebra.
By~\cite[Corollary~3.7]{scow2021ramsey}, this semi-retraction witnesses the transfer of the Ramsey property between the two structures.
Some other closely related questions remain open, e.g., whether every pre-adjunction between the ages of two structures (with embeddings as morphisms) induces a semi-retraction between the structures themselves~\cite[Question~6.6 and Theorem~6.5]{BARTOŠOVÁ_SCOW_2024}. 

According to Barto\v{s}ov\'{a} and Scow~\cite[Question~4.3]{BARTOŠOVÁ_SCOW_2024} a Ramsey transfer result via a pre-adjunction between the random ordered graph and the canonically ordered atomless Boolean algebra has previously been obtained by Ma\v{s}ulovi\'{c}~\cite{masulovic_pre-adjunctions_2018} (though this fact was originally phrased in a different language).
It would be interesting to understand how exactly Ma\v{s}ulovi\'{c}'s pre-adjunction (\cite[Theorem~3.3]{masulovic_pre-adjunctions_2018}) compares to the one induced by our semi-retraction through Theorems~6.3 and~6.4 in~\cite{BARTOŠOVÁ_SCOW_2024}.

\bibliographystyle{plain}
\bibliography{REFERENCES} 

\end{document}